\documentclass[12pt]{amsart}
\usepackage{amssymb,amsthm,amsmath,epsfig,latexsym,calc,tikz}
\usepackage{verbatim}

\newcommand{\red}[1]{{\color{red}#1}}

\numberwithin{equation}{section}
\newtheorem{theorem}{Theorem}[section]
\newtheorem{lemma}[theorem]{Lemma}

\newtheorem{corollary}[theorem]{Corollary}

\theoremstyle{definition}
\newtheorem{definition}[theorem]{Definition}
\newtheorem{construction}[theorem]{Construction}
\newtheorem{remark}[theorem]{Remark}
\newtheorem{example}[theorem]{Example}

\begin{document}

\title[Partial coloring, vertex decomposability, and SCM simplicial complexes]{Partial coloring, vertex decomposability,
and sequentially Cohen-Macaulay simplicial complexes}

\thanks{\red{The final version of this paper appeared in Journal of Commutative Algebra, Vol 7, (2015), pp. 337--352.  This version of the paper corrects
    a small error in the statement in Theorem 4.6.  Subsequent results are
    still correct, although their statements and proofs need some small changes.
    We thank Yan Gu for bringing this to our attention.
    To highlight the changes from the journal version, all changes are in
    RED.
}}

\author{Jennifer Biermann}
\address{Department of Mathematics and Statistics, 451A Clapp Lab, 
Mount Holyoke College, South Hadley, MA 01075, USA}
\email{jbierman@mtholyoke.edu}

\author{Christopher A. Francisco}
\address{Department of Mathematics, Oklahoma State University,
401 Mathematical Sciences, Stillwater, OK 74078}
\email{chris@math.okstate.edu}
\urladdr{http://www.math.okstate.edu/~chris}

\author{Huy T\`ai H\`a}
\address{Tulane University, Department of Mathematics,
6823 St. Charles Ave., New Orleans, LA 70118}
\email{tha@tulane.edu}
\urladdr{http://www.math.tulane.edu/~tai}

\author{Adam Van Tuyl}
\address{Department of Mathematical Sciences,
Lakehead University,
Thunder Bay, ON P7B 5E1, Canada}
\email{avantuyl@lakeheadu.ca}
\urladdr{http://flash.lakeheadu.ca/~avantuyl}

\keywords{simplicial complex, vertex decomposable, whiskers, sequentially
Cohen-Macaulay}
\subjclass[2000]{05E45, 05A15, 13F55}

\begin{abstract}
In attempting to understand how combinatorial modifications alter algebraic
properties of monomial ideals, several authors have investigated the process of
adding ``whiskers'' to graphs.
In this paper, we study a similar construction to build
a simplicial complex $\Delta_\chi$
from a coloring $\chi$ of a subset
of the vertices of $\Delta$, and give necessary and sufficient conditions
for this construction to produce vertex decomposable simplicial complexes.
We apply this work to strengthen and give new proofs
about sequentially Cohen-Macaulay edge
ideals of graphs.
\end{abstract}

\maketitle


\section{Introduction}

Square-free monomial ideals are intimately connected to combinatorics. This connection
raises the natural question: how do changes in combinatorial structures affect
algebraic properties of associated square-free monomial ideals? In \cite{V},
Villarreal investigates the process of adding whiskers to a finite simple graph,
and (citing also Fr\"oberg, Herzog, and Vasconcelos) proves that the edge ideal
of a graph with whiskers added to every vertex is always Cohen-Macaulay.
To add a \emph{whisker} to a vertex, one adds an additional vertex and an edge
between the old vertex and the new one.

Generalizing Villarreal's work, in \cite{FH} the second and the third authors
studied additions of whiskers to subsets of the vertices that produce
sequentially Cohen-Macaulay edge ideals. The configuration of the whiskers,
not the number, determines when the resulting ideals are sequentially
Cohen-Macaulay, demonstrating the subtlety of the problem. The techniques
 in \cite{FH} are mostly algebraic, focusing on when the cover ideals
are componentwise linear, a property which is equivalent to the Alexander
dual being sequentially Cohen-Macaulay.

In a different direction, several authors have used methods from combinatorial topology
to study similar phenomena. The primary combinatorial
object in these efforts is the independence complex of a graph,
the simplicial complex whose Stanley-Reisner
ideal coincides with the edge ideal of the graph. For instance,
Woodroofe \cite{W} and Dochtermann and Engstr\"om \cite{DE}
use combinatorial topology to prove that the
independence complex of a chordal graph is vertex decomposable, implying that
the edge ideal is sequentially Cohen-Macaulay.
Dochtermann and Engstr\"om \cite{DE} also show that
the independence complex of a completely whiskered graph is a pure vertex decomposable
simplicial complex, and consequently, Cohen-Macaulay, thus giving a
combinatorial topological proof of Villarreal's result. Cook and Nagel \cite{CN}
use full clique-whiskering, a technique that begins by partitioning the
vertex set of a graph into cliques. For each of these cliques, one adds a
new vertex and connects it to each vertex in the clique. Cook and Nagel
 prove that fully clique-whiskered graphs are vertex-decomposable
\cite[Theorem 3.3]{CN}; when the cliques in the partition each consist of a
single vertex, this recovers the results of Villarreal and Dochtermann-Engstr\"om.

The first and fourth authors take a blended approach in \cite{BVT}
to extend these results about independence complexes of graphs to all simplicial complexes.
Starting with any coloring $\chi$ of the vertices of $\Delta$,  they construct
a new simplicial complex $\Delta_\chi$ that is vertex decomposable.  The whiskering construction
of Villarreal and the clique-whiskering technique of Cook and Nagel \cite{CN}
become special instances of this construction.

The constructions in \cite{BVT, CN, V} always result in \emph{pure} vertex
decomposable simplicial complexes (and thus, Cohen-Macaulay complexes). The
algebraic results of \cite{FH} are therefore not a consequence of these results because the
corresponding independence
complex associated to the partially whiskered graph is not necessarily pure.

In this paper, in the spirit of \cite{FH}, we extend the construction in
\cite{BVT} to partial whiskerings of simplicial complexes.  We start with
a partial coloring $\chi$ of $\Delta$ (see Definition \ref{partialcolor}) and use $\chi$ and
$\Delta$ to build a new simplicial complex $\Delta_{\chi}$  (see Construction \ref{deltaChi}),
which we call a partially whiskered simplicial complex.
We give a necessary and sufficient condition for a
partially whiskered simplicial complex to be vertex decomposable.

\begin{theorem}[Theorem \ref{vertexDecomposable}] \label{thm.intro}
Let $\Delta$ be a simplicial
complex on the vertex set $V$, and let
$W$ be a subset of $V$.  Let $\chi$ be the $s$-coloring of $\Delta|_W$ given by
$W= W_1 \cup \dots \cup W_s$.
Then $\Delta_\chi$ is vertex decomposable if and only if
${\rm link}_\Delta(\mu)|_{\overline{W}}$ is vertex
decomposable for every face $\mu$ of $\Delta$ such that  $\mu \subseteq W$.
\end{theorem}

Theorem~\ref{vertexDecomposable} has a number of consequences. Corollary~\ref{numericalBound}
shows that when $W = V$,
$\Delta_\chi$ is always vertex decomposable, thus recovering the main result of \cite{BVT}.
Corollary~\ref{numericalBound} also gives the analog to the numerical bound
of the second and third authors \cite{FH} for graphs. Namely, if one has a simplicial
complex with $n$ vertices, and $|W| \geq n-3$, one
gets a vertex decomposable simplicial complex.
As in the case of graphs,
this bound is sharp (see Example~\ref{e.numericalBound}).

In Section 4, we apply Theorem~\ref{vertexDecomposable} to study
edge ideals of graphs.  In particular, we get necessary and sufficient conditions
for a whiskered graph to be vertex decomposable (see Theorem \ref{whiskeredgraphs}).
This result yields Corollary \ref{resultFH},  a
new proof for \cite[Theorem 3.3]{FH} specifying which configurations of whiskers
force the edge ideal to be sequentially Cohen-Macaulay;
this provides the combinatorial approach to the results
of \cite{FH} sought in \cite{DE}.  We also use
Theorem ~\ref{vertexDecomposable} to classify which whiskered bipartite graphs
are sequentially Cohen-Macaulay (see Theorem \ref{classifybipartite}).


Our paper is organized as follows.  In Section 2, we recall the relevant
background.  In Section 3 we present the main theorems
and derive some of their consequences.  Section 4 applies our results
to independence complexes of graphs.

{\bf Acknowledgements.}
We used CoCoA \cite{C}, Macaulay2 \cite{Mt} and the package
{\it SimplicialDecomposabilty} by David Cook II \cite{Ct} for our computer
experiments. Francisco is partially supported by a grant from the Simons Foundation,  
\#199124. H\`a is partially supported by NSA grant H98230-11-1-0165. Van Tuyl
acknowledges the support of NSERC.


\section{Background}

We recall the relevant background
on simplicial complexes.

\begin{definition} A finite {\it simplicial complex} $\Delta$ on a
finite vertex set $V$ is a
collection of subsets of $V$ with the
property that if $\sigma \in \Delta$ and $\tau \subseteq \sigma$,
then $\tau \in \Delta$.  The elements of $\Delta$ are called {\it faces}.
The maximal faces of $\Delta$, with respect to inclusion, are
the {\it facets}.
\end{definition}

The vertex sets of our simplicial complexes will
be either the set $\{x_1, \dots, x_n\}$
or the set $\{x_1, \dots, x_n, y_1, \dots, y_s\}$.
If $\Delta$ is a simplicial complex and $\sigma \in \Delta$, then
we say $\sigma$ has {\it dimension} $d$ if $|\sigma| = d+1$
(by convention, the empty set has dimension $-1$).
We say $\Delta$ is {\it pure} if all of its facets have the same dimension; otherwise
$\Delta$ is {\it non-pure}.
If $F_1,\ldots,F_t$ is a complete
list of the facets of $\Delta$, we sometimes write $\Delta$ as
$\Delta = \langle F_1,\ldots,F_t\rangle$.
If $\sigma \in \Delta$ is a face, then the
{\it deletion} of $\sigma$ from $\Delta$ is the simplicial complex defined by
\[
\Delta \setminus \sigma = \{ \tau \in \Delta ~|~ \sigma \not \subseteq \tau\} .
\]
The {\it link} of $\sigma$ in $\Delta$ is the simplicial complex defined by
\[
\rm{link}_{\Delta}(\sigma) = \{ \tau \in \Delta ~|~ \sigma \cap \tau =
\emptyset, \sigma \cup \tau \in \Delta\}.
\]
When $\sigma = \{v\}$, we shall abuse notation and
write $\Delta \setminus v$ (respectively ${\rm link}_{\Delta}(v)$) for
$\Delta \setminus \{v\}$ (respectively ${\rm link}_{\Delta}(\{v\})$).

We shall be particularly interested in the class of vertex decomposable
simplicial complexes.  This class was first introduced in the pure case by Provan
and Billera \cite{BP} and in the non-pure case by Bj\"orner and Wachs \cite{BW}.
We recall the non-pure version.

\begin{definition} A simplicial complex $\Delta$ on vertex set $V$
is called \emph{vertex decomposable} if
\begin{enumerate}
\item $\Delta$ is a simplex, or
\item there is some vertex $v \in V$ such that $\Delta \setminus v$ and
${\rm link}_\Delta(v)$ are vertex
decomposable and do not share any facets; such a vertex $v$ is called a
\emph{shedding} vertex of $\Delta$.
\end{enumerate}
\end{definition}

\begin{remark}\label{SCM}
When $\Delta$ is vertex decomposable, then $\Delta$ also inherits other
combinatorial and algebraic properties.  In particular, if $\Delta$ is pure
and vertex decomposable,
then $\Delta$ has a pure shelling, and its Stanley-Reisner ring $R/I_{\Delta}$
is Cohen-Macaulay.  If $\Delta$ is non-pure and vertex decomposable, then
$\Delta$ is still shellable,
in the non-pure sense of
Bj\"orner and Wachs \cite{BW}, and its Stanley-Reisner ring
$R/I_{\Delta}$ is sequentially Cohen-Macaulay (see Section 4 for a definition).
We point the reader
to the text of Herzog and Hibi \cite{HH} for a complete treatment of these ideas.
\end{remark}

For simplicial complexes $\Delta$ and $\Omega$ over disjoint vertex
sets $V$ and $U$, respectively, the \emph{join} of $\Delta$ and $\Omega$,
denoted by $\Delta \cdot \Omega$, is the simplicial complex over the vertex set
$V \cup U$, whose faces are $\{\sigma \cup \tau ~|~ \sigma \in \Delta, \tau \in
\Omega\}$.  Provan and Billera proved:

\begin{theorem}[\protect{\cite[Proposition 2.4]{BP}}] \label{thm.BP}
The join $\Delta \cdot \Omega$ is vertex decomposable if and only if both
$\Delta$ and $\Omega$ are vertex decomposable.
\end{theorem}

The property of vertex decomposability is preserved when taking a link.
The following result was first proved in  \cite[Proposition 2.3]{BP}
in the pure case;  the non-pure case follows similarly, as noted in
\cite[Theorem 3.30]{J} and \cite[Proposition 3.7]{W2}.

\begin{theorem} \label{linkvd}
If $\Delta$ is vertex decomposable,
then ${\rm link}_{\Delta}(\sigma)$ is vertex decomposable for any $\sigma \in \Delta$.
\end{theorem}

An important notion for our main construction and results in
Section \ref{Colorings} is that of a coloring of a
simplicial complex.

\begin{definition} Let $\Delta$ be a simplicial complex on the vertex set $V$
 with facets $F_1, \dots, F_t$.  An {\it s-coloring} of $\Delta$  is a
partition of the vertices  $V = V_1\cup \dots \cup V_s$
(where the sets $V_i$ are allowed to be empty) such that  $|F_i \cap V_j| \leq 1$
for all
$1 \leq i \leq t, 1 \leq j \leq s$.
We will sometimes write that $\chi$ is an $s$-coloring of $\Delta$
to mean $\chi$ is a specific partition of $V$ that gives an $s$-coloring of $\Delta$.
If there exists an $s$-coloring of $\Delta$, we say that
$\Delta$ is \emph{s-colorable}.
\end{definition}

Note that the definition of an $s$-coloring is equivalent to an $s$-coloring (in the graph theoretic sense) of the 1-skeleton of the simplicial complex.

\begin{example}  If $\Delta$ is a simplicial complex on $|V| = n$ vertices,
then $\Delta$ is $n$-colorable; indeed, we take our coloring
to be $V = \{x_1\} \cup \{x_2\} \cup \cdots \cup \{x_n\}$.
\end{example}

In this paper, we are interested in the case in which a subset of the
vertices of a simplicial complex is colored, or equivalently, in the coloring
of an induced subcomplex.

\begin{definition}
Let $\Delta$ be a simplicial complex on the vertex set
$V$, and let $W \subseteq V$.  The \emph{restriction of $\Delta$ to $W$}
(or equivalently, the \emph{induced subcomplex over $W$}) is the
subcomplex
\[
\Delta |_W = \{\sigma \in \Delta ~|~ \sigma \subseteq W\}.
\]
\end{definition}

\begin{definition}\label{partialcolor}
Let $\Delta$ be a simplicial complex on the vertex set $V$, and let
$W \subseteq V$.  If $\chi$ is an $s$-coloring of the restriction $\Delta|_W$,
then we call $\chi$ a \emph{partial coloring} of $\Delta$.  We call the vertices
in $W$ the \emph{colored vertices} of $\Delta$ and those in $\overline{W} = V\setminus W$ the
\emph{non-colored vertices}.
\end{definition}

\begin{example}\label{examplepartial}
Let $\Delta = \langle x_1x_2x_3, x_2x_4\rangle$, and let 
$W = \{x_1, x_2, x_4\}$.  Then $\Delta|_W = \langle x_1x_2,x_2x_4 \rangle$.
Then a $2$-coloring $\chi$ of $\Delta|_W$ is 
given by $W = \{x_1, x_4\} \cup \{x_2\}$.   So $\chi$ is a partial
coloring of $\Delta$, where the vertices of $W$ are the colored vertices,
and $\{x_3\}$ is a non-colored vertex.
\end{example}


\section{Partial colorings and vertex decomposability}\label{Colorings}

Given any simplicial complex $\Delta$ with a partial coloring $\chi$,
we introduce a construction to make a new simplicial complex
$\Delta_\chi$.  Our main result gives necessary and
sufficient conditions for $\Delta_\chi$ to be
vertex decomposable. We begin with the construction of $\Delta_\chi$.

\begin{construction}\label{deltaChi}  Let $\Delta$ be a simplicial complex
on the vertex set $V = \{x_1, \dots, x_n\}$, and let $W$ be a subset of
$V$.  Let $\chi$ be an $s$-coloring of $\Delta|_W$ given by
$W= W_1 \cup \dots \cup W_s$. Define  $\Delta_\chi$ to be the simplicial complex on the vertex set
$\{x_1, \dots, x_n, y_1, \dots, y_s\}$ with
faces $\sigma \cup \tau$, where $\sigma$ is a face of $\Delta$, and $\tau$ is
any subset of $\{y_1,\dots, y_s\}$ such that for all
$y_j \in \tau$, we have $\sigma \cap W_j = \emptyset$.  Note in particular that since $\emptyset$ is a face of $\Delta$, $\{y_1,y_2,\dots, y_s\}$ is always a face of $\Delta_\chi$.
\end{construction}

The first and fourth authors recently studied Construction \ref{deltaChi} in
\cite{BVT} in the case that $W=V$;
this case also appears in \cite{F} and implicitly in a proof in \cite{BFS}.  We
call the process of adding a new vertex $y_i$ for a color class $W_i$
\emph{whiskering}, and call the resulting complex $\Delta_\chi$ the
\emph{(partially) whiskered} simplicial complex.

\begin{example}
Consider the simplicial complex
$\Delta = \langle x_1x_2x_3, x_2x_4\rangle$ of Example \ref{examplepartial}.
If $W = \{x_1, x_2, x_4\}$,  take $\chi$ to be the 
coloring of $\Delta|_W$ given by $W = \{x_1, x_4\} \cup \{x_2\}$.  
Then $\Delta_\chi = \langle x_3y_1y_2, x_4y_2, x_1x_3y_2, x_2x_3y_1, x_2x_4, x_1x_2x_3\rangle$.  $\Delta$ and $\Delta_\chi$ are shown in Figure \ref{Coloring example}.

\begin{figure}[h]
\begin{tikzpicture}[thick, scale = 2]
\filldraw[fill= black!10!white, xshift = -3.5cm]  (0,0) -- (225:1) -- (315:1) -- cycle;
\draw [xshift = -3.5cm] (315:1) -- (45:2);
\filldraw[white, xshift = -3.5cm] (45:2) circle (1.5pt);
\draw[black, xshift = -3.5cm] (45:2) circle (1.5pt) node[anchor=west] {$x_4$};
\filldraw[black!50!white, xshift = -3.5cm ] (315:1) circle (1.5pt);
\draw[black, xshift = -3.5cm] (315:1) circle (1.5pt) node[anchor=west] {$x_2$};
\filldraw[white, xshift = -3.5 cm] (225:1) circle (1.5pt);
\draw[black, xshift = -3.5cm] (225:1) circle (1.5pt) node[anchor=east] {$x_1$};
\filldraw[black, xshift = -3.5cm] (0,0) circle (1.5pt) node[anchor=west] {$x_3$};

\filldraw[fill= black!10!white]  (45:1) -- (135:1) -- (225:1) -- (315:1) -- cycle;
\draw (45:1) -- (0,0) -- (225:1);
\draw(135:1) -- (0,0) -- (315:1);
\draw (315:1) -- (45:2) -- (135:1);
\filldraw[black] (0,0) circle (1.5pt) node[anchor=west] {$x_3$};
\filldraw[white] (45:1) circle (1.5pt);
\draw[black] (45:1) circle (1.5pt) node[anchor=west] {$y_1$};
\filldraw[black!50!white] (135:1) circle (1.5pt);
\draw[black] (135:1) circle (1.5pt) node[anchor=east, black] {$y_2$};
\filldraw[white] (225:1) circle (1.5pt);
\draw[black] (225:1) circle (1.5pt) node[anchor=east] {$x_1$};
\filldraw[black!50!white] (315:1) circle (1.5pt);
\draw[black] (315:1) circle (1.5pt) node[anchor=west] {$x_2$};
\filldraw[white] (45:2) circle (1.5pt);
\draw[black] (45:2) circle (1.5pt) node[anchor=west] {$x_4$};
\end{tikzpicture}
\caption{ \label{Coloring example} The simplicial complexes $\Delta$ (left) and $\Delta_\chi$ (right).}
\end{figure}
\end{example}

\begin{remark}
To forestall any potential confusion,  ``whiskering a vertex,'' as first
defined in \cite{V},
referred to adding a new additional vertex to a graph and joining the new and
old vertices by an edge.
This operation was defined in terms of the finite simple graph.  However, this
procedure also results in a change
in the independence complex of the graph (see Section 4 for details).  In our definition,
when we use the term ``whiskering'', we are generalizing the operation that
changes the independence complex,
not the graph.
\end{remark}

Our main result is necessary and sufficient conditions for $\Delta_{\chi}$ to be
vertex decomposable.

\begin{theorem}\label{vertexDecomposable}  Let $\Delta$ be a simplicial
complex on the vertex set $V$, and let $W$ be a subset of
$V$.  Let $\chi$ be the $s$-coloring of $\Delta|_W$ given by
$W= W_1 \cup \dots \cup W_s$.
Then $\Delta_\chi$ is vertex decomposable if and only if
${\rm link}_\Delta(\mu)|_{\overline{W}}$ is vertex
decomposable for every face $\mu$ of $\Delta$ such that  $\mu \subseteq W$.
\end{theorem}

\begin{remark}\label{onecondition}
The special case $\mu=\emptyset$ in Theorem~\ref{vertexDecomposable}
is instructive. Because ${\rm link}_\Delta(\emptyset) = \Delta$, the
link hypothesis imposes the condition that $\Delta|_{\overline{W}}$
is vertex decomposable.
\end{remark}

\begin{proof}[Proof of Theorem~\ref{vertexDecomposable}]
$(\Leftarrow)$
We proceed by induction on the number of vertices of $\Delta$.  The base case
is the empty simplicial complex $\Delta = \{\emptyset\}$.  In this case the
only partial coloring of $\Delta$ is $W = W_1\cup \dots\cup W_s$, where all the
$W_i$ are empty.  Then $\Delta_\chi$ is the simplex
$\langle \{y_1,\dots, y_s\}\rangle$,
which is vertex decomposable.

Now let $\Delta$ be a simplicial complex on vertex set
$V = \{x_1, \dots, x_n\}$, and let $W \subseteq V$.
If $W = \emptyset$,
then $\Delta_{\chi} = \Delta \cdot \langle \{y_1,\ldots, y_s\} \rangle$,
the join of two simplicial complexes. The link hypothesis with $\mu=\emptyset$
implies that $\Delta|_{\overline{\emptyset}} = \Delta$ is vertex decomposable.
It then follows from Theorem~\ref{thm.BP} that
$\Delta_{\chi}$ is vertex decomposable.
We can therefore assume that $\Delta \neq \{\emptyset\}$
and $W \neq \emptyset$.

Let $w \in W$.   After relabelling, we will assume that $w \in W_1$.
To prove that $\Delta_\chi$ is vertex decomposable, we will
show that $\Delta_\chi \setminus w$ and ${\rm link}_{\Delta_\chi} (w)$
are both vertex decomposable.

Recall that the faces of $\Delta_\chi$ are
of the form $\sigma \cup \tau$, where $\sigma$ is a face of $\Delta$ and
$\tau \subseteq \{y_1, \dots y_s\}$ such that if $y_j \in \tau$, then
$W_j \cap \sigma = \emptyset$.  Note that
\begin{align*}
\Delta_\chi \setminus w
&= \{\sigma \cup \tau \in \Delta_\chi ~|~ w \notin \sigma \cup \tau\}\\
	&= \{\sigma \cup \tau \in \Delta_\chi ~|~ w \notin \sigma\}\\
	&= (\Delta \setminus w)_{\chi'}
\end{align*}
where $\chi'$ is the partial coloring $(W_1 \setminus \{w\}) \cup W_2
\cup \cdots \cup W_s$
of $\Delta \setminus w$ induced by the coloring $\chi$ of $\Delta$.
Since $w \in W$, the uncolored vertices of $\Delta$ and $\Delta \setminus w$ are
the same set $\overline{W}$.
Now let $\mu$ be a face of  $\Delta \setminus w$ such that
$\mu \subseteq (W \setminus \{w\})$.  Then since $w \notin \overline{W}$,
\begin{align*}
({\rm link}_{\Delta \setminus w}(\mu))|_{\overline{W}}
&= \{ \sigma \in (\Delta \setminus w) ~|~ \mu \cap \sigma = \emptyset,
\mu \cup \sigma \in (\Delta \setminus w)\}|_{\overline{W}}\\
	&= \{\sigma \in \Delta ~|~ w \notin \sigma, \mu \cap \sigma = \emptyset,
\mu \cup \sigma \in (\Delta \setminus w)\}|_{\overline{W}}\\
	&= \{\sigma \in \Delta ~|~ \mu \cap \sigma = \emptyset, \mu \cup \sigma
\in \Delta\}|_{\overline{W}}\\
	&= ({\rm link}_\Delta (\mu))|_{\overline{W}}.
\end{align*}
Thus we have that $({\rm link}_{\Delta \setminus w}(\mu))|_{\overline{W}}$
is vertex decomposable by hypothesis.
Therefore, since $\Delta \setminus w$ is a simplicial complex on fewer
than $n$ vertices with $\chi'$ a partial coloring on $W\setminus \{w\}$
such that $({\rm link}_{\Delta \setminus w}(\mu))|_{\overline{W}}$ is vertex
decomposable for all $\mu \subseteq (W\setminus \{w\})$,  induction implies that
$\Delta_\chi \setminus w = (\Delta \setminus w)_{\chi'}$ is vertex decomposable.

Now consider ${\rm link}_{\Delta_\chi}(w)$.  We have
\begin{align*}
{\rm link}_{\Delta_\chi}(w)
&= \{ \sigma \cup \tau \in \Delta_\chi ~|~
w \notin \sigma \cup \tau, \sigma \cup \tau \cup \{w\} \in \Delta_\chi\}\\
	&= \{\sigma \cup \tau \in \Delta_\chi ~|~w \notin \sigma, (\sigma
\cup \{w\})\cup \tau \in \Delta_\chi\}\\
	&= ({\rm link}_\Delta (w))_{\chi''},
\end{align*}
where $\chi''$ is the partial coloring of ${\rm link}_\Delta(w)$
given by $U_2 \cup \dots \cup U_s$, and
$U_j = W_j \cap \{$vertices of ${\rm link}_\Delta(w)\}$.  Because
$w \in W_1$, we need only consider $U_j$ for $j=2,\ldots,s$.
Set $U = U_2 \cup \dots \cup U_s$ (i.e., the colored vertices of ${\rm
link}_\Delta(w)$) and $\overline{U}$ to be the set of non-colored vertices of
${\rm link}_\Delta(w)$.

Note that
$\overline{U} = \overline{W} \cap \{$vertices of ${\rm link}_\Delta(w)\}$.
%
Let $\mu$ be a face of ${\rm link}_\Delta(w)$ such that
$\mu \subseteq U$.  Since $\mu \in {\rm link}_\Delta(w)$, we have
$\mu \cup \{w\} \in \Delta$.  Then
\[
{\rm link}_{{\rm link}_\Delta(w)}(\mu) = {\rm link}_\Delta(\mu \cup \{w\})
\]
so
\[
({\rm link}_{{\rm link}_\Delta(w)}(\mu))|_{\overline{U}} = ({\rm link}_{{\rm link}_\Delta(w)}(\mu))|_{\overline{W}} = ({\rm link}_\Delta(\mu \cup
\{w\}))|_{\overline{W}}.
\]
By the assumption on $\Delta$, $ ({\rm link}_\Delta(\mu \cup
\{w\}))|_{\overline{W}}$ is vertex decomposable.
Because ${\rm link}_{\Delta_\chi}(w) = ({\rm link}_\Delta (w))_{\chi''}$, and ${\rm
link}_\Delta (w)$ is a simplicial complex on fewer than $n$ vertices,
${\rm
link}_{\Delta_\chi}(w)$ is vertex decomposable by induction.

To show that $\Delta_\chi$ is vertex decomposable, all that remains is to
show that no facet of ${\rm link}_{\Delta_\chi}(w)$ is a facet of
$\Delta_\chi \setminus w$.  Let $\sigma \cup \tau$ be a facet of
${\rm link}_{\Delta_\chi}(w)$ .  Then $(\sigma \cup \{w\}) \cup \tau$ is a
face of $\Delta_\chi$, so $y_1 \notin \tau$.    On the other hand, if
$\sigma \cup \tau \in \Delta_\chi \setminus w$, then since
$\sigma \cap W_1 = \emptyset$, $\sigma \cup \tau \cup \{y_1\}$ is also a face
of $\Delta_\chi \setminus w$ and so the
link and deletion do not share any facets.

$(\Rightarrow)$  Let $\mu \in \Delta$ be a face such that $\mu \subseteq W$, and hence
$\mu \in \Delta|_W$.  Because
$\chi$ is an $s$-coloring of $\Delta|_W$, we have
$|\mu \cap W_i| \leq 1$ for $i=1,\ldots,s$.
After relabelling the $W_j$'s, we may assume that
$|\mu \cap W_i| = 1$ for $i=1,\ldots,t$,
and $|\mu \cap W_i| = 0$ for $i=t+1,\ldots,s$.

By Construction \ref{deltaChi}, we have $\mu \cup \{y_{t+1},\ldots,y_s\} \in \Delta_{\chi}$.
We now claim that
\[{\rm link}_{\Delta}(\mu)|_{\overline{W}} = {\rm link}_{\Delta_{\chi}}(\mu \cup \{y_{t+1},\ldots,y_s\}).\]
Notice that our conclusion will then follow from this claim and
Theorem \ref{linkvd} because $\Delta_\chi$ is
assumed to be vertex decomposable.

For any $\tau \in {\rm link}_{\Delta}(\mu)|_{\overline{W}}$,
we have $\tau \cup \mu \in \Delta$, $\tau \cap \mu = \emptyset$, and $\tau \cap W = \emptyset$.
By Construction \ref{deltaChi}, $\tau \cup \mu \cup \{y_{t+1},\ldots,y_s\} \in \Delta_{\chi}$
because $(\tau \cup \mu) \cap W_i = \emptyset$ for $i=t+1,\ldots,s$.  Moreover, since
$\tau \cap (\mu \cup \{y_{t+1},\ldots,y_s\}) = \emptyset$, we have $\tau \in {\rm link}_{\Delta_{\chi}}(\mu \cup \{y_{t+1},\ldots,y_s\})$.

We now consider the reverse inclusion.  Let $\tau \in
{\rm link}_{\Delta_{\chi}}(\mu \cup \{y_{t+1},\ldots,y_s\})$.  Thus, $\tau \cup \mu \cup \{y_{t+1},\ldots,y_s\}
\in \Delta_{\chi}$ and $\tau \cap (\mu \cup \{y_{t+1},\ldots,y_s\}) = \emptyset$.  Since $|\mu \cap W_i| = 1$ for $ i = 1, \dots, t$ we know that $y_i \notin \tau$ for $i = 1, \dots, t$.  Therefore $\tau \subseteq \{x_1, \dots, x_n\}$, and $\tau \cup \mu$ must be a face of $\Delta$.  Thus $\tau \in {\rm link}_\Delta(\mu)$.

Further, since $\tau \cup \mu \in \Delta$ and $|\mu \cap W_i| = 1$ for $i =1, \dots, t$, we have $|\tau \cap W_i| = 0$ for $i = 1, \dots, t$.  Since $\tau \cup \mu \cup \{y_{t+1}, \dots, y_s\} \in \Delta_\chi$ we have $|\tau \cap W_i| = 0$ for $i = t +1, \dots, s$ as well.  Therefore $\tau \in {\rm link}_\Delta(\mu)|_{\overline{W}}$.
\end{proof}

\begin{remark}\label{verticesnotenough}
Let $\Delta = \langle x_1x_2x_3x_4, x_1x_3x_4x_5, x_1x_3x_5x_6, x_1x_2x_5x_6,
x_2x_3x_6\rangle$, and let $\chi$ be the coloring given by $W = \{x_1\} \cup \{x_2\}$.
Then $\Delta|_{\overline{W}}$, ${\rm link}_\Delta(x_1)|_{\overline{W}}$, and
${\rm link}_\Delta(x_2)|_{\overline{W}}$ are all vertex decomposable. However,
${\rm link}_\Delta(\{x_1,x_2\})|_{\overline{W}}$ is not vertex decomposable, so by
Theorem ~\ref{vertexDecomposable}, neither is $\Delta_{\chi}$.
\end{remark}




We now give a bound on the number of vertices to color to ensure
that $\Delta_\chi$ is vertex decomposable.   The following corollary also recovers \cite[Theorem 3.7]{BVT} in the case where $|V \setminus W| = 0$.

\begin{corollary}\label{numericalBound}
Let $\Delta$ be a simplicial complex on vertex set $V$, $W$ a
subset of $V$, and $\chi$ a coloring of $\Delta|_W$.  If $|V\setminus W|  \leq 3$,
then $\Delta_\chi$ is vertex decomposable.
\end{corollary}

\begin{proof}
All simplicial complexes on three or fewer vertices are vertex decomposable.
Since $|\overline{W}| = |V \setminus W| \leq 3$,
${\rm link}_\Delta(\mu)|_{\overline{W}}$ is vertex decomposable for any
$\mu \in \Delta$ such that $\mu \subseteq W$.  Thus, by
Theorem \ref{vertexDecomposable}, $\Delta_\chi$ is vertex decomposable.
\end{proof}

The previous corollary is an analog of a bound of the second and third
authors \cite[Corollary 3.5]{FH}.
The numerical bound on the cardinality of $\overline{W}$ in
Corollary \ref{numericalBound} is sharp:

\begin{example}\label{e.numericalBound}
Let $\Delta = \langle x_1x_2x_3, x_3x_4x_5\rangle$, $W = \{x_3\}$ and
$\chi$ be the coloring of $\Delta|_W$ given by $W = W_1 = \{x_3\}$,
so $|\overline{W}| = 4$.  Then
${\rm link}_{\Delta}(\emptyset)=\Delta|_{\overline{W}}  = \langle x_1x_2, x_4x_5\rangle$, which is not vertex
decomposable. Thus $\Delta$ with this coloring does not fit the conditions of
Theorem \ref{vertexDecomposable}.  Indeed,
$\Delta_\chi = \langle x_4x_5y_1, x_1x_2y_1, x_3x_4x_5, x_1x_2x_3\rangle$ is
not vertex decomposable.
\end{example}

As noted in \cite{FH}, it is not necessarily the number
of whiskers but rather their configuration that determines whether the resulting ideal is sequentially Cohen-Macaulay. A similar phenomenon occurs in our
setting. In particular, given any coloring of $\Delta$,
if $\chi$ is its restriction to all but one color class, then $\Delta_\chi$ is
vertex decomposable.

\begin{corollary}
Let $\Delta$ be a simplicial complex on the vertex set $V = \{x_1,\ldots,x_n\}$,
and let $\chi$ be an $s$-coloring of $\Delta$ given by
$V = V_1 \cup \cdots \cup V_s$.  For each $i =1,\ldots,s$, let
$\chi_i$ be the induced partial coloring of $\Delta|_{Y_i}$
given by $Y_i = V_1 \cup \cdots V_{i-1} \cup
V_{i+1} \cup \cdots V_s$.  Then $\Delta_{\chi_i}$ is vertex
decomposable for each $i=1,\ldots,s$.
\end{corollary}

\begin{proof}
It suffices to prove the statement for $i=1$.   Let $Y_1 = V_2 \cup \cdots \cup
V_s$ be the induced partial coloring of $\Delta$ given by $\chi_1$.  Then
$\overline{Y_1} = V_1$.  Since $\chi$ is a coloring, if $\sigma \in \Delta$
and $\sigma \subseteq V_1$, then $|\sigma| \leq 1$.  Then for any $\mu \subseteq Y$,
${\rm link}_{\Delta}(\mu)|_{V_1}$ is either the simplicial complex $\{\emptyset\}$
or a zero-dimensional simplex. Because these simplicial complexes are
vertex decomposable, Theorem \ref{vertexDecomposable}
implies the desired result.
\end{proof}


\section{Sequentially Cohen-Macaulay edge ideals}

We round out this paper by applying Theorem \ref{vertexDecomposable} to
edge ideals of graphs.  In particular, we give a new
proof of \cite[Theorem 3.3]{FH} and classify when a whiskered bipartite graph
is sequentially Cohen-Macaulay.

We recall some terminology.
Let $G = (V,E)$ be a finite simple graph. We say $W \subseteq V$ is an
{\it independent set} if for all $e \in E$, $e \cap W \neq e$.  We can
form a simplicial complex from the independent sets of $G$:

\begin{definition} Let $G$ be a finite simple graph.  The {\it independence complex}
of $G$, denoted ${\rm Ind}(G)$, is the simplicial complex
${\rm Ind}(G) = \{W ~|~ \mbox{$W$ is an independent set of $G$}\}.$
\end{definition}

If $V = \{x_1,\ldots,x_n\}$, we can identify the vertices of $G$ with the variables
of $R = k[x_1,\ldots,x_n]$.  The {\it edge ideal} of $G$, denoted $I(G)$, is
the quadratic squarefree monomial ideal generated by the monomials $x_ix_j$ whenever
$\{x_i,x_j\} \in E$.  The edge ideal of $G$ is the Stanley-Reisner ideal
of ${\rm Ind}(G)$, i.e, $I(G) = I_{{\rm Ind}(G)}$.

We recall the definition of sequentially Cohen-Macaulay modules as it pertains to $I(G)$.

\begin{definition}  The graph $G$ is {\it sequentially Cohen-Macaulay} if $R/I(G)$
is sequentially Cohen-Macaulay; that is, there exists a finite filtration of
graded $R$ modules $0=M_0 \subset M_1 \subset \cdots \subset M_t=R/I(G)$ such
that $M_i/M_{i-1}$ is Cohen-Macaulay for each $i \in \{1,\ldots,t\}$, and for all $i \in \{1,\dots,t-1\}$, $\dim M_i/M_{i-1} < \dim M_{i+1}/M_i$.
The ideal $I(G)$ is a {\it sequentially Cohen-Macaulay edge ideal} if $G$ is
sequentially Cohen-Macaulay.
\end{definition}

If we specialize the general theory of vertex decomposable simplicial complexes,
we have the following link between these objects.

\begin{theorem}\label{vdimpliesscm}
If ${\rm Ind}(G)$ is vertex decomposable, then $G$ is sequentially Cohen-Macaulay.
\end{theorem}

\noindent
Given a subset $S \subseteq V$, we denote by $G \cup W(S)$
the graph obtained by adding whiskers to all the vertices of $S$.  After
relabelling, we can always assume $S = \{x_1,\ldots,x_s\}$.   The following lemma,
which is simply applying the definitions,
describes the connection between the whiskered graph $G \cup W(S)$ and
Construction \ref{deltaChi}.

\begin{lemma}\label{whiskerconnection}
Let $G$ be a finite simple graph on the vertex set $V$,
and let $S \subseteq V$.  Let $\chi$ be the $s$-coloring of ${\rm Ind}(G)|_S$
given by $S = \{x_1\} \cup \cdots \cup \{x_s\}$.  Then
${\rm Ind}(G \cup W(S)) = {\rm Ind}(G)_{\chi}.$
\end{lemma}

\begin{remark}
One can also recover the clique-starring and clique-whiskering techniques of
Woodroofe \cite{W} and Cook and Nagel \cite{CN} from Construction~\ref{deltaChi},
coloring all vertices and allowing each coloring class to have more than one vertex.
\end{remark}

When restricted to independence complexes of graphs,
Theorem \ref{vertexDecomposable} gives us necessary and sufficient conditions for
a whiskered graph to have vertex decomposable independence complex.
\red{For any subset $\mu \subseteq V$, let
  $$N[\mu] = \mu \cup \{x \in V ~|~ \mbox{$x$ is adjacent to some $y \in \mu$}
  \};$$
this is the closed neighborhood of $\mu$.}
    Below, $G\setminus \red{N[\mu]}$ denotes a graph $G$ with the vertices
    of $\red{N[\mu]}$ and adjacent edges removed.  
For any subset $W \subseteq V$, we use $G|_W$ to denote
the induced subgraph of $G$ on $W$, i.e., the graph with vertices the elements of $W$ and edge set
$\{e \in E ~|~ e \subseteq W\}$.

\begin{theorem} \label{whiskeredgraphs}
Let ${\rm Ind}(G)$ be the independence complex of a graph $G$ on a vertex set
$V$, and let $S \subseteq V$.  Then ${\rm Ind}(G \cup W(S))$ is vertex decomposable
if and only if ${\rm Ind}((G \setminus \red{N[\mu]})|_{\overline{S}})$ is vertex decomposable
for all $\mu \in {\rm Ind}(G)$ with $\mu \subseteq S$.
\end{theorem}

\begin{proof}
By Lemma \ref{whiskerconnection},
${\rm Ind}(G \cup W(S)) = {\rm Ind}(G)_{\chi}$ for the $s$-coloring $\chi$ of
${\rm Ind}(G)_S$ given by $S = \{x_1\} \cup \cdots \cup \{x_s\}$.  On the other hand,
for any $\mu \subseteq S$, one can show that
\[{\rm link}_{{\rm Ind}(G)}(\mu)|_{\overline{S}} = {\rm Ind}(G \setminus
\red{N[\mu]})|_{\overline{S}}
= {\rm Ind}((G \setminus \red{N[\mu]})|_{\overline{S}}).\]
Thus the statement is simply restricting Theorem \ref{vertexDecomposable} to
independence complexes.
\end{proof}

Recall that a graph $G$ is {\it chordal} if
it has no induced cycles of length $\geq 4$.
The independence
complexes of chordal graphs are particularly nice:

\begin{theorem}[{\cite[Theorem 4.1]{DE}}{\cite[Corollary 7(2)]{W}}]  \label{chordal}
If $G$ is chordal graph, then ${\rm Ind}(G)$ is vertex
decomposable.
\end{theorem}

We now show how Theorem \ref{whiskeredgraphs} not only allows us to give
a new proof of \cite[Theorem 3.3]{FH} but also strengthen it.
The original conclusion of \cite[Theorem 3.3]{FH} is that the associated edge
ideals are sequentially Cohen-Macaulay. This now follows from Corollary
\ref{resultFH} and Theorem \ref{vdimpliesscm}.

\begin{corollary} \label{resultFH}
Let $G$ be a finite simple graph, and let $S \subseteq V$.  Suppose that $G
\setminus S$, the induced subgraph over the vertices $V \smallsetminus S$,
is either a chordal graph or the five cycle $C_5$.  Then
${\rm Ind}(G \cup W(S))$ is vertex decomposable.  In particular, $G \cup W(S)$
is sequentially Cohen-Macaulay.
\end{corollary}

\begin{proof}
First assume that $G \setminus S$ is a chordal graph.
Let $V = \{x_1,\ldots,x_n\}$ and suppose, after relabelling,
$S = \{x_1,\ldots,x_s\}$.  Let $\chi$ be the $s$-coloring
of ${\rm Ind}(G)|_S$ given by $S = \{x_1\} \cup \cdots \cup \{x_s\}$.
For any $\mu \subseteq S$,
$(G \setminus \red{N[\mu]})|_{\overline{S}}$ is an induced subgraph of $G|_{\overline{S}}$,
so it is chordal.  By Theorem \ref{chordal},
${\rm Ind}((G \setminus \red{N[\mu]})|_{\overline{S}})$ is vertex decomposable.  Now apply
Theorem \ref{whiskeredgraphs}.

The proof for the case $G \setminus S$ is a five-cycle is similar because
the independence complex of a five-cycle is vertex decomposable,
as are any induced subgraphs.
\end{proof}

\begin{remark}
The proof of Corollary \ref{resultFH} only requires
that all the induced subgraphs of $G \setminus S$
have the property that their independence complexes
be vertex decomposable.   Woodroofe \cite{W} has shown
that all graphs whose only induced cycles are either three-cycles
or five-cycles have this property;  this family
contains the family of graphs listed in the above corollary.
Our statement of Corollary \ref{resultFH}
was chosen to highlight the connection to the work of \cite{FH}.
\end{remark}

In the case of bipartite graphs, we can strengthen
our results and classify exactly when a whiskered bipartite graph is
sequentially Cohen-Macaulay.
Recall that  we say that a graph $G = (V, E)$ is \emph{bipartite} if there
exists a partition $V = V_1 \cup V_2$ such that for all $e \in E$ we have
$e \cap V_1 \neq \emptyset$ and $e \cap V_2 \neq \emptyset$.  We then
need the following result of the fourth author.

\begin{theorem}[{\cite[Theorem 2.10]{VT}}]\label{bipartiteSCM}
Let $G$ be a bipartite graph.  Then ${\rm Ind}(G)$ is vertex decomposable if and only
if $G$ is sequentially Cohen-Macaulay.
\end{theorem}

We then have the following classification.

\begin{theorem}\label{classifybipartite}
Let $I(G)$ be the edge ideal of a bipartite graph $G$ on a vertex set
$V$ and let $S \subseteq V$.  Then $G \cup W(S)$ is sequentially
Cohen-Macaulay
if and only if $(G \setminus \red{N[\mu]})|_{\overline{S}}$ is sequentially
Cohen-Macaulay
for all $\mu \in {\rm Ind}(G)$ with $\mu \subseteq S$.
\end{theorem}

\begin{proof}
When $G$ is a bipartite graph, the graph $G \cup W(S)$
will also be bipartite for any $S \subseteq V$.
In addition, any subgraph of the form $(G \setminus \red{N[\mu]})_{\overline{S}}$ will also be
bipartite.

So, Theorem \ref{bipartiteSCM} implies $G \cup W(S)$
is sequentially Cohen-Macaulay if and only if ${\rm Ind}(G \cup W(S))$ is
vertex decomposable.  But by Theorem \ref{whiskeredgraphs},
${\rm Ind}(G \cup W(S))$ is
vertex decomposable if and only if
${\rm Ind}(G  \setminus \red{N[\mu]})_{\overline{S}}$ is vertex
decomposable for all $\mu \in {\rm Ind}(G)$ with $\mu \subseteq S$.  But again by
Theorem \ref{bipartiteSCM}, this can happen if and only if
 $(G  \setminus \red{N[\mu]})_{\overline{S}}$ is sequentially
Cohen-Macaulay for all $\mu \in {\rm Ind}(G)$ with $\mu \subseteq S$.
\end{proof}

\begin{remark}
It is natural to ask if Theorem \ref{classifybipartite}
holds for all graphs, not just bipartite graphs.  Examining the proof of
Theorem \ref{classifybipartite}, we might be able to find a counterexample
if there exists a graph $G$ that is sequentially Cohen-Macaulay, but
${\rm Ind}(G)$ is not vertex decomposable, i.e., if the converse of 
Theorem \ref{vdimpliesscm} is
false. This is indeed the case. As pointed out in \cite[Example 4.4]{FH}, 
we can construct a graph $G$ from a minimal triangulation of the real projective plane 
such that $G$ is Cohen-Macaulay over a field $k$ if and only if the characteristic 
of $k$ is not 2. In particular, over a field of characteristic 0, $G$ is Cohen-Macaulay, 
but ${\rm Ind}(G)$ is not vertex decomposable.
\end{remark}


\end{document}